\documentclass[12pt]{amsart}
\usepackage[colorlinks=true,pagebackref,hyperindex,citecolor=green,linkcolor=red]{hyperref}
\usepackage{amsmath}
\usepackage{verbatim}
\usepackage{amsfonts}
\usepackage{amssymb}
\usepackage{color}
\usepackage[all]{xy}  
\usepackage{enumerate}
\usepackage[top=1in, bottom=1in, left=0.9in, right=0.9in]{geometry}
\usepackage{mathrsfs}

\usepackage{stmaryrd}
\usepackage{bbold}
\renewcommand{\k}{\mathbb{k}}
\usepackage{cleveref}
\usepackage{soul}

\theoremstyle{definition}
\newtheorem{thm}{Theorem}[section]
\newtheorem{theoremx}{Theorem}
\numberwithin{equation}{section}

\newtheorem{corollary}[thm]{Corollary}
\newtheorem{lemma}[thm]{Lemma}
\newtheorem{proposition}[thm]{Proposition}

\newtheorem{notation}[thm]{Notation}

\theoremstyle{definition}
\newtheorem{definition}[thm]{Definition}

\newtheorem{example}[thm]{Example}

\newtheorem{remark}[thm]{Remark}

\newtheoremstyle{TheoremNum}
        {8pt}{8pt}              %%% space between body and thm
        {\upshape}                      %%% Thm body font
        {}                              %%% Indent amount (empty = no indent)
        {\bfseries}                     %%% Thm head font
        {.}                             %%% Punctuation after thm head
        {.5em}                             %%% Space after thm head
        {\thmname{#1}\thmnote{ \bfseries #3}}%%% Thm head spec
  \theoremstyle{TheoremNum}

\newcommand{\m}{\mathfrak{m}}
\newcommand{\n}{\mathfrak{n}}

\newcommand{\NN}{\mathbb{N}}
\newcommand{\ZZ}{\mathbb{Z}}

\newcommand{\FF}{\mathbb{F}}

\newcommand{\PP}{\mathbb{P}}
\newcommand{\lr}[1]{\langle {#1} \rangle}

\newcommand{\cC}{\mathscr{C}}
\newcommand{\cV}{\mathcal{V}}

\newcommand{\cF}{\mathcal{F}}

\newcommand{\cP}{\mathcal{P}}

\newcommand{\lrphi}[1]{\lr{\varphi_{#1}}}
\newcommand{\nCM}{\operatorname{nCM}}

\newcommand{\sdim}{\operatorname{sdim}}

\newcommand{\Spec}{\operatorname{Spec}}
\newcommand{\Depth}{\operatorname{depth}}
\newcommand{\Hom}{\operatorname{Hom}}
\newcommand{\End}{\operatorname{End}}
\newcommand{\Ext}{\operatorname{Ext}}

\newcommand{\M}{\mathcal{M}}
\newcommand{\N}{\mathcal{N}}

\newcommand{\InjDim}{\operatorname{inj.dim}}
\newcommand{\Supp}{\operatorname{Supp}}

\newcommand{\Ann}{\operatorname{Ann}}
\newcommand{\Ass}{\operatorname{Ass}}	
	
\newcommand{\depth}{\operatorname{depth}}

\newcommand{\Proj}{\operatorname{Proj}}
	
\newcommand{\Ht}{\operatorname{ht}}

\newcommand{\ls}{\leqslant}%
\newcommand{\gs}{\geqslant}
\renewcommand{\leq}{\leqslant}
\renewcommand{\geq}{\geqslant}
\newcommand{\ds}{\displaystyle}

\newcommand{\p}{\mathfrak{p}}
\newcommand{\q}{\mathfrak{q}}

\newcommand{\ann}{\operatorname{ann}}
\newcommand{\ck}[1]{{#1}^{\vee}}

\renewcommand{\a}{\mathfrak{a}}
\newcommand{\OO}{\mathcal{O}}

\renewcommand{\H}{{\rm H}}

\renewcommand{\mod}{\operatorname{mod}}

\title{Local cohomology and Lyubeznik numbers of $F$-pure rings}

\author[De Stefani]{Alessandro De Stefani}
\address{Department of Mathematics, University of Nebraska, Lincoln, NE 68588-0130, USA}
\email{adestefani2@unl.edu}

\author[Grifo]{Elo\'isa Grifo}
\address{Department of Mathematics, University of Virginia, Charlottesville, VA 22904-4135, USA}
\email{eloisa.grifo@virginia.edu}

\author[N\'u\~nez-Betancourt]{Luis N\'u\~nez-Betancourt${^*}$}
\address{Centro de Investigaci\'on en Matem\'aticas, Guanajuato, GTO, Mexico.}
\email{luisnub@cimat.mx}

\thanks{{$^*$}The third author was partially supported by the NSF Grant 1502282}


\dedicatory{Dedicated to Professor~Craig~Huneke on the occasion of his birthday.}

\begin{document}
\maketitle

\begin{abstract}
In this article, we study certain local cohomology modules over $F$-pure rings. We give sufficient conditions for the vanishing of some Lyubeznik numbers, derive a formula for computing these invariants when the $F$-pure ring is standard graded and, by its means, we provide some new examples of Lyubeznik tables. We study associated primes of certain $\Ext$-modules, showing that they are all compatible ideals. Finally, we focus on properties that Lyubeznik numbers detect over a globally $F$-split projective variety.
\end{abstract}

%%%%%%%%%%%%%%%%%%%%%%%%%%%%%%%%%%%%%%%%%%%%%%%%%%%%%%%%%
\section{Introduction}
%%%%%%%%%%%%%%%%%%%%%%%%%%%%%%%%%%%%%%%%%%%%%%%%%%%%%%%%%
Local cohomology modules play an important role in the study of algebraic and geometric properties of rings. These modules are usually not finitely generated, and difficult to manage. For an equicharacteristic regular local ring, $(S,\m,\k)$, and an ideal $I \subseteq S$, the local cohomology modules $\H^i_I(S)$ have a finite set of associated primes and finite Bass numbers \cite{HunekeSharp,LyuDMod}.  These results are derived from the structure of $\H^i_I(S)$  as a $D$-module in characteristic zero, and as an $F$-module in positive characteristic. In this manuscript, we focus on the case where  $\k$ has prime characteristic $p>0$, $[\k^{1/p}:\k]<\infty$, and $R=S/I$ is an $F$-pure ring.

We say that $R$ is $F$-pure if the the Frobenius map splits. The singularities of an $F$-pure ring can be quite severe. For instance, $F$-pure rings need not be normal or Cohen-Macaulay. However, the local cohomology modules $\H^i_\m(R)$ enjoy desirable finiteness properties \cite{MaRFmod}. 
If $R$ is written as a quotient of an $n$-dimensional regular local ring $S$ by an ideal $I$, these good properties are reflected in the modules $\H^{n-i}_I(S)$, by Matlis duality \cite{LyuFmod}. In this paper, we further this study by showing that the associated primes of $\H^{n-i}_I(S)$ behave well with respect to the Cartier algebra $\cC(R) = \bigoplus_e \cC^e(R)$. See Section \Cref{SecAssPrimes} for the definition of $\cC(R)$.

\begin{theoremx}[\Cref{compatibility filtration}]\label{MainAssPrimes}
Let $(S,\n,\k)$ be an $F$-finite regular local ring. Let $I \subseteq S$ be an ideal such that $R=S/I$ is $F$-pure. If $\p \in\Ass_S \H^{n-i}_I(S)$, then $\phi(\p R)\subseteq \p R$ for every $e \in \NN$, and every $\phi \in \cC^e(R)$.
\end{theoremx}
\Cref{compatibility filtration} is actually more general, as it deals with the primes associated to the simple $F$-module factors of $\H^{n-i}_I(S)$, which correspond to the simple $R\langle F\rangle$-module factors of $\H^i_\m(R)$ (see \Cref{SecAssPrimes} for definitions and details).

Since the introduction of the Lyubeznik numbers of a ring $R$ \cite{LyuDMod}, many interactions with the geometry and topology of $\Spec(R)$ have been found. For instance, the Lyubeznik numbers detect if the punctured spectrum of $R$ is connected \cite{Walther2}, and highly connected \cite{W}. These numbers also relate to singular cohomology \cite{GarciaSabbah}, and \'etale cohomology \cite{B-B}. 
Furthermore, the Lyubeznik numbers of a Stanley-Reisner ring associated to a simplicial complex are topological invariants of its geometric realization \cite{AMYLyuNum}. 
We refer to \cite{SurveyLyuNum} for a survey of these topics. The Lyubeznik  number of $R$ with respect to $i,j$ is defined as 
$$
\lambda_{i,j}(R)=\dim_\k \Ext^{i}_S(\k,\H^{n-j}_I(S)).
$$
It turns out that these numbers are finite \cite{HunekeSharp,LyuDMod}, and do not depend on the chosen representation $R \cong S/I$ \cite{LyuDMod}.
As a consequence of the techniques involved in the proof of \Cref{MainAssPrimes}, we give a sufficient condition for the vanishing of certain Lyubeznik numbers.

\begin{theoremx}[see Theorem \ref{first main result}]\label{MainVanishing}
Let $(S,\n,\k)$ be a complete $F$-finite local ring.
Let $I \subseteq R$ be an ideal such that $R=S/I$ is an $F$-pure ring, and let $\sdim(R)$ denote the splitting  dimension of $R$.
If either $i<\sdim(R)$ or $j<\sdim(R)+1$, then $\lambda_{i,j}(R)=0$.
\end{theoremx}

Unfortunately, the Lyubeznik numbers are difficult to compute. 
Even finding out whether or not $\lambda_{i,j}(R)$ vanishes is a difficult computational problem (see \cite{KatzmanZhang} for progress in this direction).
Formulas for Lyubeznik numbers have been obtained previously for rings of small dimension \cite{Walther}, and Stanley-Reisner rings \cite{YStr,Montaner,AM-Proc,AMV}. We give a formula for the Lyubeznik numbers of standard graded $F$-pure rings. This recovers and extends, in positive characteristic, the formula for Stanley-Reisner rings \cite[Theorem 1.1]{YStr}. 

\begin{theoremx}[see Theorem \ref{lyu formula}]\label{MainFormula}
Let $\k$ be a field of characteristic $p>0$, and $S=\k\left[ x_1, \ldots, x_n \right]$, with $\deg(x_i)=1$. Let $I$ be a homogeneous ideal in $S$ be such that $R = S/I$ is $F$-pure. Then
$$\lambda_{i,j}(R) = \dim_\k \left( \Ext^{n-i}_S \left(  \Ext^{n-j}_S \left( R, S \right), S \right) \right)_0,$$
where the last subscript denotes the degree zero part.
\end{theoremx}

This theorem gives an algorithm to compute Lyubeznik numbers of previously unknown examples. For instance, in Subsection \ref{SubSecExamples} we use the previous formula to compute Lyubeznik numbers of some binomial ideals.

In positive characteristic, there is a notion of Lyubeznik numbers for projective varieties. Since Zhang \cite{WZ-projective} showed that these invariants are well-defined, it is natural to ask what aspects of a projective variety $X$ are tracked by $\lambda_{i,j}(X)$. In our last main result, we investigate the meaning of the Lyubeznik numbers for a globally $F$-split projective variety.

\begin{theoremx}[{see Theorem \ref{ThmHighest} and \ref{ThmSheaf}, and Proposition \ref{PropDuality}}]\label{MainProj}
Let $X$ be a globally $F$-split, Cohen-Macaulay and equidimensional projective variety over $\k$.
Let $d$ be the dimension of $X$ and $t$ denote the number of connected components of $X \times_\k\overline{\k}$. Then:
\begin{enumerate}
\item $\lambda_{0,1}(X) = t-1$; 
\item $\lambda_{d+1,d+1}(X)=t$;
\item $\lambda_{0,j}(X)=\dim_\k\left(\H^j(X,\OO_X)\right)$ for $2 \ls j \ls d$;
\item $\lambda_{j,d+1}(X)=\dim_\k\left(\H^{d+2-j}(X,\OO_X)\right)$ for $2 \ls j \ls d$; 
\item all other Lyubeznik numbers  vanish.
\end{enumerate}
\end{theoremx}
We note that smooth projective varieties are Cohen-Macaulay, and so one can apply the previous theorem to  such varieties, provided they are globally $F$-split.

%%%%%%%%%%%%%%%%%%%%%%%%%%%%%%%%%%%%%%%%%%%%%%%%%%%%%%%%%
\section{Vanishing of Lyubeznik numbers for $F$-pure rings}
%%%%%%%%%%%%%%%%%%%%%%%%%%%%%%%%%%%%%%%%%%%%%%%%%%%%%%%%%
\subsection{Background on local cohomology and methods in positive characteristic}
%%%%%%%%%%%%%%%%%%%%%%%%%%%%%%%%%%%%%%%%%%%%%%%%%%%%%%%%%

Throughout this paper, $R$ will denote a commutative Noetherian ring with unity. Let $I \subseteq R$ be an ideal generated by $\underline{f} = f_1, \ldots, f_t \in R$. The \v{C}ech complex on $f_1, \ldots, f_t$, denoted $\check{C}^{\bullet}(\underline{f};R)$, is the complex
\[
\ds \xymatrix{0 \ar[r] & R \ar[r] & \bigoplus_i R_{f_i} \ar[r] & \bigoplus_{i,j }R_{f_i f_j} \ar[r] & \ldots \ar[r] & R_{f_1 \ldots f_{t}} \ar[r] & R \ar[r] & 0}.
\]
More explicitly, the modules involved are of the form $\check{C}^{i}(\underline{f};R) = \bigoplus_{1 \leqslant j_1 < \ldots < j_i \leqslant t} R_{f_{j_1}\cdots f_{j_i}}$, and the differential map between two consecutive elements of the complex is a sum of localization maps, with appropriate sign choices. Given an $R$-module $M$, we define $\check{C}^{\bullet}(\underline{f};M) := \check{C}^{\bullet}(\underline{f};R) \otimes_R M$. The \emph{$i$-th local cohomology} of $M$ with support in $I$ is the $i$-th cohomology of the complex $\check{C}^{\bullet}(\underline{f};M)$, which we write as $\H^i_I(M) = \H^i \left(\check{C}^{\bullet}(\underline{f};M) \right)$. This is well-defined, meaning that it does not depend on the choice of generators $f_1,\ldots,f_t$ of the ideal $I$. 

Now assume that $(R,\m,\k)$ is a local ring, and let $M$ be a non-zero $R$-module, not necessarily finitely generated. The depth of $M$ is defined as $\Depth(M) = \inf\{j \in \ZZ \mid \H^j_\m(M) \ne 0\}$. Alternatively, one can see that $\Depth(M) = \inf\{j \in \ZZ \mid \Ext^j(\k,M) \ne 0\}$. A related concept is that of Bass numbers. Given a prime ideal $\p \in \Spec(R)$, the $j$-th Bass number of $M$ with respect to $\p$ is $\mu^j(\p,M) := \dim_{\kappa(\p)}\Ext^j_{R_\p}(\kappa(\p),M_\p)$, where $\kappa(\p) = R_\p/\p R_\p$. In particular, $\Depth(M) = \inf\{j \in \ZZ \mid \mu^j(\m,M) \ne 0\}$.

The Lyubeznik numbers of a ring correspond to Bass numbers of certain local cohomology modules. Specifically, let $(R,\m,\k)$ be a local ring, and suppose that there exists a regular local ring $(S,\n, \k)$ that maps onto $R$. Then, we can write $R \cong S/I$ for some ideal $I$ in $S$. Let $n$ be the Krull dimension of $S$. The {\it Lyubeznik number of $R$ with respect to $i,j \in \NN$} is defined as
$$\lambda_{i,j} (R) := \dim_{\k} \Ext^i_S \left(\k, \H^{n-j}_I (S) \right).$$
This definition depends only on $R$, $i$ and $j$, meaning that it does not depend on the choices of $S$ and $I$. For a survey on Lyubeznik numbers, see \cite{SurveyLyuNum}.

Let $R$ be a ring of characteristic $p>0$. We say that $R$ is \emph{$F$-finite} if it is a finitely generated module over itself via the action of the Frobenius map $F\!: R \longrightarrow R$. We say that $R$ is \emph{$F$-pure} if the Frobenius map $F\!: R \longrightarrow R$ is a pure homomorphism, that is, if for all $R$-modules $M$, the induced map $F \otimes 1 \!\!: R \otimes M \longrightarrow R \otimes_R M$ is injective. 
If $R$ is $F$-finite, then $R$ is $F$-pure if and only if the map $F\!\!:R\to R$ splits. Throughout this article, we denote by $F_*R$ the $R$-module structure on $R$ induced by the Frobenius endomorphism. To avoid potential confusion, we will denote the elements of $F_*R$ by $F_*r$, for each $r \in R$. We note that $R$ is $F$-finite if and only if $F_*R$ is a finite $R$-module.  From this perspective, $F$-purity is equivalent to the property that $R \to F_*R$ splits as a map of $R$-modules, if $R$ is $F$-finite \cite[Corollary 5.3]{HRFpurity}.

\begin{definition}[\cite{AE}]
Let $(R,\m,\k)$ be a local ring, which is $F$-finite and $F$-pure. For every positive integer $e$, define
$$I_e(R) := \left\lbrace r \in R \, | \, \varphi \left( F^e_*r \right) \in \m \textrm{ for all } \varphi \in \Hom_R \left( F^e_*R, R \right) \right\rbrace.$$
The \emph{splitting prime} of $R$ is defined as $\cP(R) := \bigcap_e I_e(R)$,  and $  \dim \left( R / \cP(R) \right)$ is denoted by $\sdim(R)$, and called the \emph{splitting dimension} of $R$.
\end{definition}

Observe that, if $f \notin \cP(R)$, then the $R$-module map $R \longrightarrow F^e_*R \stackrel{\cdot F^e_*f}{\longrightarrow} F^e_*R$ splits for all $e \gg 0$.

%%%%%%%%%%%%%%%%%%%%%%%%%%%%%%%%%%%%%%%%%%%%%%%%%%%%%%%%%
\subsection{Depth and vanishing theorems}
%%%%%%%%%%%%%%%%%%%%%%%%%%%%%%%%%%%%%%%%%%%%%%%%%%%%%%%%%

In this section, we exploit properties of $F$-pure rings to conclude that certain Lyubeznik numbers must vanish. We start with a useful lemma, that relates the depth of a module with that if its localizations, under suitable assumptions.

\begin{lemma} \label{depth formula} Let $(R,\m,k)$ be a Cohen-Macaulay local ring, with canonical module $\omega_R$, and $M$ be a finitely generated $R$-module. Assume that the set $\mathcal{A} := \bigcap_i \Supp_R(\Ext^i_R(M,\omega_R))$ is not empty, where the intersection is taken over all indices $i$ such that $\Ext^i_R(M,\omega_R) \ne 0$. Then, for all primes $\p \in \mathcal{A}$, we have 
\[
\depth(M) = \depth(M_\p) + \dim(R/\p).
\] 
\end{lemma}
\begin{proof}
Let $d = \dim(R)$. We may assume, without loss of generality, that $R$ is complete. By local duality, we have
 \[
\ds \Depth(M) = \min\{j \in \ZZ \mid \H^j_\m(M) \ne 0\} = \min\{j \in \ZZ \mid \Ext^{d-j}_R(M,\omega_R) \ne 0\}. 
 \]
Now let $\p \in \mathcal{A}$. By assumption, we obtain that $\Ext^i_R(M,\omega_R) = 0$ if and only if $\Ext^i_R(M,\omega_R)_\p \cong \Ext^i_{R_\p}(M_\p,(\omega_R)_\p) = 0$. Furthermore, the localization $(\omega_R)_\p$ is a canonical module for $R_\p$, that is, $(\omega_R)_\p \cong \omega_{R_\p}$ \cite[Remark 2.2 i)]{HoHuOmega}. Putting these facts together we get
\[
\ds \Depth(M) = \min\{j \in \ZZ \mid \Ext^{d-j}_{R_\p}(M_\p,\omega_{R_\p})\ne 0\}.
\]
We denote by $\widehat{R_\p}$ the completion of the local ring $R_\p$ at its maximal ideal. We have that

\begin{align*}
\Ext^{d-j}_{R_\p}(M_\p,\omega_{R_\p})\neq 0
& \Longleftrightarrow \Ext^{\dim(R_\p) + \dim(R/\p) -j}(M_\p,\omega_{R_\p})\neq 0, \hbox{ because }d=\dim(R_\p) + \dim(R/\p)\\
& \Longleftrightarrow \Ext^{\dim(R_\p) + \dim(R/\p) -j}(M_\p,\omega_{R_\p})\otimes_{R_\p}\widehat{R_\p} \ne 0\\
& \Longleftrightarrow \H^{j-\dim(R/\p)}_{\p \widehat{R_\p} }(M_\p\otimes_{R_\p}\widehat{R_\p} ) \ne 0,\hbox{ by local duality}\\
& \Longleftrightarrow \H^{j-\dim(R/\p)}_{\p R_\p}(M_\p) \ne 0 .\\
\end{align*}

Therefore,
\[
\ds \Depth(M) = \min\{j \in \ZZ \mid \H^{j-\dim(R/\p)}_{\p R_\p}(M_\p) \ne 0\} = \Depth(M_\p) + \dim(R/\p).
\]
\end{proof}

In what follows, let $(R,\m,\k)$ be a local ring of prime characteristic $p>0$. We denote by $\ck{(-)}$ the Matlis duality functor, that is, the functor $\Hom_R(-,E_R(\k))$, where $E_R(\k)$ denotes the injective hull of the residue field $\k$.
\begin{definition} We say that a functor $G$, from the category of $R$-modules to itself, commutes with the action of Frobenius if there is a functorial isomorphism $G(F_*-) \cong F_*G(-)$. 
\end{definition}

Here are some examples of such functors:
\begin{itemize}
\item The identity functor trivially commutes with the action of Frobenius,
\item The Matlis duality functor commutes with the action of Frobenius \cite[Lemma 5.1]{BB-CartierMod}.
\item  If $R$ is a quotient of a regular local ring $S$, then the functor $G(-) = \Ext^i_S(-,S)$ commutes with the action of Frobenius \cite[Proposition 5.7 and Lemma 5.8]{W} for all $i \in \NN$.
\end{itemize}

\begin{proposition} \label{vanishing}
Let $(S,\n,\k)$ be a complete $F$-finite regular local ring, $I\subseteq R$ be an ideal such that $R=S/I$ is an $F$-pure ring. Let $\mod(R)$ denote the category of finitely generated $R$-modules, and $G \!: \mod(R) \to \mod(R)$ be an additive functor that commutes with the action of Frobenius. If $\cP(R)$ denotes the splitting prime of $R$, then
\[
\ds \Depth(G(R))= \sdim(R)+\Depth(G(R)_{\cP(R)}).
\]
\end{proposition}
\begin{proof}
Let $\m=\n/I$ be the maximal ideal of $R$, and let $j \in \NN$ be such that $\H^j_\m(G(R)) \ne 0$. By local duality, we have that $\ck{\H^j_\m(G(R))} \cong \Ext^{n-j}_S(G(R),S) \ne 0$, where $n=\dim(S)$. 
If $Q$ is a prime ideal in $S$ of height strictly less than $n-j$, then $\Ext^{n-j}_S(G(R),S)_Q \cong \Ext^{n-j}_{S_Q}(G(R)_Q,S_Q)$ must be zero, since $\InjDim_{S_Q}(S_Q) = \Ht_S(Q)$.
Therefore, 
$$\dim(\Ext^{n-j}_S(G(R),S)) \ls n-(n-j) = j.$$
We now claim that $\Ann_R(\Ext^{n-j}_S(G(R),S)) \subseteq \cP(R)$. Let $f \in \Ann_R(\Ext^{n-j}_S(G(R),S))$, and assume that $f \notin \cP(R)$. This means that for all $e \gg 0$, the inclusion $R \cdot ( F^e_*f ) \subseteq F^e_*R$ splits as a map of $R$-modules, where $R \cdot (F^e_*f)$ is the cyclic $R$-submodule of $F^e_*R$ generated by $F^e_*f$. Note that $f \cdot \Ext^{n-j}_S(G(R),S) = 0$ is equivalent to the statement that the image of the map induced by applying $\Ext^{n-j}_S(-,S)$ to the map $G(R) \stackrel{\cdot f}{\longrightarrow} G(R)$ is zero. Applying $F^e_*(-)$ to $\Ext^{n-j}_S(G(R),S)\stackrel{\cdot f}{\longrightarrow} \Ext^{n-j}_S(G(R),S)$, we then obtain that the image of 
\[
\ds \Ext^{n-j}_S(G(F^e_*R),S) \cong F^e_*(\Ext_S^{n-j}(G(R),S)) \stackrel{\cdot F^e_*f}{\longrightarrow} F^e_*(\Ext_S^{n-j}(G(R),S)) \cong \Ext^{n-j}_S(G(F^e_*R),S)
\]
is still zero, where we used that both $\Ext^{n-j}_S(-,S)$ and $G(-)$ commute with the action of Frobenius. However, the composition $R \hookrightarrow F^e_*R \stackrel{\cdot F^e_*f}{\longrightarrow} F^e_*R$ of maps of $R$-modules splits by assumption. If $G$ is covariant, the induced map 
\[
\xymatrixcolsep{5mm}
\xymatrixrowsep{2mm}
\xymatrix{
\Ext_S^{n-j}(G(F^e_*R),S) \ar[rr]^-{\cdot F^e_*f} &&  \Ext_S^{n-j}(G(F^e_*R),S) \ar[rr] && \Ext_S^{n-j}(G(R),S)
}
\]
is surjective. This gives a contradiction, since we are assuming that the first map is zero, and that $\Ext^{n-j}_S(G(R),S) \ne 0$. If $G$ is contravariant, the induced map 
\[
\xymatrixcolsep{5mm}
\xymatrixrowsep{2mm}
\xymatrix{
\Ext_S^{n-j}(G(R),S) \ar[rr] && \Ext_S^{n-j}(G(F^e_*R),S) \ar[rr]^-{\cdot F^e_*f} &&  \Ext_S^{n-j}(G(F^e_*R),S) 
}
\]
is injective, which is again a contradiction, since the rightmost map is zero.
Either way, we reach the desired contradiction, and hence we proved that $\Ann_R(\Ext_S^{n-j}(R,S)) \subseteq \cP(R)$. In particular, this implies that $\cP(R) \subseteq \bigcap_i \Supp_R(\Ext^i_S(G(R),S))$, where the intersection is taken over the indices $i$ for which $\Ext^i_S(G(R),S) \ne 0$. The statement in the proposition now follows as a direct application of Lemma \ref{depth formula}, noting that $\omega_S \cong S$ since $S$ is a regular local ring.
\end{proof}
As a consequence of Proposition \ref{vanishing}, we obtain a result that gives the vanishing of $\lambda_{i,j}(R)$ for $j < \sdim(R)$.
\begin{corollary} \label{CorollVanishing}
Let $(S,\n,k)$ be a complete $F$-finite regular local ring, $I \subseteq R$ an ideal such that $R=S/I$ is an $F$-pure ring. For an integer $j \in \NN$, if $\H^j_I(S)\neq 0$, then $\Depth(\H^j_I(S))\gs \sdim(R)$.
\end{corollary}
\begin{proof}
Consider the functor $G = \Ext^j_S(-,S)$, which commutes with the action of Frobenius. It follows from Proposition \ref{vanishing} that $\Depth(\Ext^j_S(R,S)) \gs \sdim(R)$. In particular, $\mu^i(\n,\Ext^j_S(R,S)) = 0$ for all $i < \sdim(R)$.
Therefore, we have that $\mu^i(\n,\H^j_I(S)) = 0$ for all $i < \sdim(R)$ \cite[Theorem 2.1]{HunekeSharp}, which shows that $\Depth(\H^j_I(S)) \gs  \sdim(R)$.
\end{proof}

\begin{thm}\label{first main result}
Let $(S,\n,\k)$ be a complete $F$-finite regular local ring of dimension $n$.
Let $I\subseteq R$ be an ideal such that $R=S/I$ is an $F$-pure ring. Then, $\lambda_{i,j}(R)=0$ if either $i<\sdim(R)$ or $j<\sdim(R)+1$. Furthermore, if $R$ satisfies Serre's condition $(S_k)$, we have that $\lambda_{i,j}(R)=0$ if $j<\min\{\sdim(R)+k,\dim(R)\}$.
\end{thm}

\begin{proof}
By Corollary \ref{CorollVanishing}, we directly obtain that $\lambda_{i,j}(R) = \dim_\k(\Ext^i_S(\k,\H^{n-j}_I(S))) = 0$ for all $i<\sdim(R)$. To prove the vanishing of the remaining Lyubeznik numbers, it is enough to show the last claim. In fact, $F$-pure rings are reduced, hence $R$ automatically satisfies Serre's condition $(S_1)$. Note that $\H^j_\m(R) = 0$ if and only if $\H^{n-j}_I(S) = 0$ \cite[Theorem 3.1]{LyuVan}, as a consequence of the $F$-purity of $R$. It follows from Proposition \ref{vanishing} that $\H^{n-j}_I(S) = 0$ for all $j < \sdim(R) + \Depth(R_{\cP(R)})$. Since 
$$\Depth(R_{\cP(R)}) \geqslant \min \{ k, \Ht(\cP(R)) \},$$
we have that $\H^{n-j}_I(S)=0$ for all 
$$j  < \min\{\sdim(R) + k, \sdim(R)+\Ht(\cP(R))\}.$$
If $k=1$, then either $\min\{k,\Ht(\cP(R))\} = 1$, or $\cP(R)$ is a minimal prime of $R$. In the first case, we clearly have $\H^{n-j}_I(S)=0$ for all $j<\min\{k+\sdim(R),\dim(R)\}$. On the other hand, if $\cP(R)$ is a minimal prime, then we must actually have $\cP(R)=0$, and $R$ is strongly $F$-regular. In this case, $R$ is Cohen-Macaulay, hence $\H^{n-j}_I(S)=0$ for all $j<\dim(R)$. If $k \geqslant 2$, then $R$ is catenary and equidimensional. In fact, $F$-finite rings are catenary, and catenary $(S_2)$ rings are equidimensional. Then $\Ht(\cP(R))+\sdim(R) = \dim(R)$ holds, completing the proof.
\end{proof}

%%%%%%%%%%%%%%%%%%%%%%%%%%%%%%%%%%%%%%%%%%%%%%%%%%%%%%%%%%%%%%%%%%%%%%%%%%
\section{Properties of associated primes of local cohomology modules} \label{SecAssPrimes}
%%%%%%%%%%%%%%%%%%%%%%%%%%%%%%%%%%%%%%%%%%%%%%%%%%%%%%%%%%%%%%%%%%%%%%%%%%%%

Let $(S,\n,\k)$ be a complete $F$-finite regular local ring of dimension $n$, and $I\subseteq R$ be an ideal such that $R=S/I$ is an $F$-pure ring. Observe that $\Ass_S(\H^i_I(S)) = \Ass_S(\Ext_S^i(R,S))$. In fact, $\Ass_S(\H^i_I(S)) \subseteq \Ass_S(\Ext^i_S(R,S))$ always holds  \cite[Corollary 2.3]{HunekeSharp}. The converse containment holds because $R$ is $F$-pure \cite[Theorem 1.3]{AnuragUliPure}. It can be seen from the methods used in the proof of Proposition \ref{vanishing}, applied to the identity functor, that every minimal prime over $\Supp_R(\Ext_S^i(R,S))$ is contained in the splitting prime $\cP(R)$. In this section, we extend this result to any associated prime of $\Ext_S^i(R,S)$ and, more generally, to every prime appearing as a factor in a prime filtration of $\Ext^i_S(R,S)$. In fact, we will show that all these prime ideals are compatible, which implies that they are contained in $\cP(R)$. We first recall some results about $F$-modules and $R\lr{F}$-modules, which are key techniques used in this section.

%%%%%%%%%%%%%%%%%%%%%%%%%%%%%%%%%%%%%%%%%%%%%%%%%%%%%%%%%%%%%%%%%%%%%%%%%%
\subsection{Background on $F$-modules and $R\langle F \rangle$-modules} \label{SubSecFmodRFmod}
%%%%%%%%%%%%%%%%%%%%%%%%%%%%%%%%%%%%%%%%%%%%%%%%%%%%%%%%%%%%%%%%%%%%%%%%%%%%

For more information regarding $F$-modules, $R\lr{F}$-modules, and their relations we refer to Lyubeznik's foundational paper \cite{LyuFmod}.

Let $S$ be a Noetherian ring of characteristic $p>0$. For an integer $e>0$, let $F^e_*S$ denote $S$ viewed as a module over itself under restriction of scalars via $F^e\!:S \to S$, the $e$-th iteration of the Frobenius map. We let $\cF_S^e(-)$ denote the Peskine-Szpiro Frobenius functor on $S$, that is defined as follows: given any $S$-module $M$, $\cF_S^e(M)$ is the base change $M \otimes_S F^e_*S$, followed by the natural identification of $F^e_*S$ with $S$ on the second component of the tensor product.

\begin{example} Let $S$ be a Noetherian ring of characteristic $p>0$. 
\begin{itemize}
\item  $\cF^e_S(S)=S$ for all $e$. 
\item If $I \subseteq S$ is an ideal, we have $\cF^e_S(S/I) = S/I^{[p^e]}$, where $I^{[p^e]} = (x^{p^e} \mid x \in I)$ is the $p^e$-th Frobenius power of the ideal $I$. 
\item More generally, let $S^m \stackrel{\phi}{\longrightarrow} S^n \to M \to 0$ be a presentation of a finite $S$-module $M$, with $\phi$ an $n \times m$ matrix. Then $S^m \stackrel{\phi^{[p^e]}}{\longrightarrow} S^n \to \cF^e_S(M) \to 0$ is a presentation of $\cF^e_S(M)$, where $\phi^{[p^e]}$ is the matrix whose entries are the $p^e$-th powers of the entries of $\phi$.
\end{itemize}
\end{example}

We now recall the main definitions and some basic concepts related to $F$-modules. We will refer to them as $F^e$-modules, as we will allow powers of the Peskine-Szpiro Frobenius functor. As already noted in \cite[Remark 5.6 a.]{LyuFmod}, most of the results about $F$-modules, appropriately restated, readily generalize to this context. 

In what follows, $(S,\n,\k)$ is assumed to be an $F$-finite regular local ring of prime characteristic $p>0$. In this case, the functor $\cF_S^e(-)$ is exact for all positive integers $e$ \cite[Theorem 2.1]{Kunz}.
\begin{definition}{\cite[Definition 1.1]{LyuFmod}} An $F^e_S$-module is a pair $(\M, \theta)$, where $\M$ is an $S$-module, and $\theta \in \Hom_S(\M,F^e_S(\M))$ is an isomorphism. We call $\theta$ the structure morphism of $\M$. A morphism of $F^e_S$-modules is an $S$-linear map $f\!:\M \to \N$ such that the following diagram commutes:
\[
\xymatrix{
\M \ar[d] \ar[r] & \N \ar[d] \\ 
\cF^e_S(\M) \ar[r] & \cF^e_S(\N)
}
\]
We say that $\beta\!:M \to \cF^e_S(M)$ is a generating morphism for an $F^e_S$-module $(\M,\theta)$ if $\M$ is the direct limit of the inductive system
\[
\xymatrix{
M \ar[d]^-\beta \ar[r]^-\beta & \cF^e_S(M) \ar[d]^-{\cF^e_S(\beta)} \ar[r]^-{\cF^e_S(\beta)} \ar[r] & \cF^{2e}_S(M) \ar[d]^-{\cF_S^{2e}(\beta)} \ar[r] & \cdots \\
\cF^e_S(M) \ar[r]^-{\cF^e_S(\beta)} & \cF^{2e}_S(M)  \ar[r]^-{\cF_S^{2e}(\beta)} & \cF^{3e}_S(M) \ar[r] & \cdots
}
\]
and $\theta$ is the map induced on the direct limits by the vertical arrows in this diagram. An $F^e_S$-module is called $F^e_S$-finite if it has a generating morphism $\beta\!:M \to \cF^e_S(M)$, with $M$ a finitely generated $S$-module. Furthermore, if $\beta$ is injective, we call $M$ a root of $\M$, and $\beta$ a root morphism of $\M$.
\end{definition}

\begin{remark}  If $S$ is complete, any $F^e_S$-finite module $\M$ has a unique minimal root, which is contained in every root of $\M$ \cite[Theorem 3.5]{LyuFmod}.
\end{remark}
Meaningful examples of $F^e_S$-modules include local cohomology modules $\H^i_I(S)$, where $I$ is an ideal in $S$. A generating morphism is given by the map $\Ext^i_S(S/I,S) \to F_S(\Ext^i_S(S/I,S)) \cong \Ext^i_S(S/I^{[p]},S)$, where the isomorphism is due to the fact that $F_S(-)$ is exact.

We now recall the definition of $R\langle F \rangle$-modules, and adapt it to higher order Frobenius actions. Let $(R,\m,k)$ be a complete local ring of characteristic $p>0$.

\begin{definition} A Frobenius action on an $R$-module $M$ is an additive map $\varphi:\! M \to M$ such that $\varphi(rm) = r^p\varphi(m)$ for all $r \in R$ and $m\in M$. Equivalently, this corresponds to an $R$-linear map $M \to \varphi_*M$. More generally, an $e$-th Frobenius action on $M$ is an additive map $\varphi_e\!:M \to M$ such that $\varphi_e(rm) = r^{p^e} \varphi_e(m)$ for all $r \in R$ and $m\in M$.
\end{definition}

$R$-modules with an $e$-th Frobenius action $\varphi_e\!:M \to M$ correspond to left $R\lrphi{e}$-modules, where $R\lrphi{e}$ denotes the subring of $\End_\ZZ(R,R)$ generated by $R$ and an $e$-th Frobenius action $\varphi_e$. The ring $R$ acts on itself by multiplication, and the interaction with $\varphi_e$ is subject to the rule $\varphi_er = r^{p^e} \varphi_e$, for all $r \in R$. 
Since in this section we only deal with left $R\lrphi{e}$-modules, we will omit the word ``left''.

\begin{definition} We say that an $R\lrphi{e}$-module is cofinite if it is Artinian as an $R$-module.
\end{definition}

Let $F\!:R \to R$ be the Frobenius endomorphism on $R$, $f \in R$, and $i \in \NN$. The map $fF^e\!:R \to R$ induces a morphism $fF^e\!:\H^i_\m(R) \to \H^i_\m(R)$ on local cohomology modules, which is an $e$-th Frobenius action. Thus, the modules $\H^i_\m(R)$ are cofinite $R\lrphi{e}$-modules.

\begin{definition} Let $M$ be an $R\lrphi{e}$-module. We say that $M$ is anti-nilpotent if for any $R \lrphi{e}$-submodule $W \subseteq M$ the induced Frobenius action $\varphi_e: M/W \to M/W$ is injective.
\end{definition}

The following lemma is a natural generalization of the techniques and of some of the results obtained by Ma  \cite{MaRFmod}, regarding good cohomological properties of $F$-pure rings.

\begin{lemma} \label{generalization Linquan} Let $(R,\m,k)$ be an $F$-finite local ring, and let $\varphi_e \!:R \to R$ be an $e$-th Frobenius action on $R$. Assume that the corresponding $R$-module map $R \to \varphi_{e*}R$ splits. Then $\H^i_\m(R)$, with the induced $R\lrphi{e}$-module structure, is anti-nilpotent.
\end{lemma}
\begin{proof}
For sake of completeness, we provide a proof here. The ideas are very similar to those used by Ma \cite{MaRFmod}. Set $\varphi:=\varphi_e$, and let $\Phi:\H^i_\m(R) \to \H^i_\m(R)$ be the map induced by $\varphi$ on local cohomology. Let $W \subseteq \H^i_\m(R)$ be a $\Phi$-stable submodule, and assume that $\Phi(y) \in W$ for some $y \in \H^i_\m(R)$. It is sufficient to show that $y \in R$-span$\langle \Phi(y),\Phi^2(y),\Phi^3(y),\ldots\rangle$, since such a module is contained in $W$ by assumption. Since $\H^i_\m(R)$ is Artinian, there exists an integer $e' \geqslant 0$ such that $\Phi^{e'}(y) \in R$-span$\langle \Phi^{e'+1}(y),\Phi^{e'+2}(y),\ldots\rangle$. Therefore, we have $\Phi^{e'}(y) = \sum_{i=1}^n r_i \Phi^{e'+i}(y)$ for some $n$ and some $r_i \in R$. By assumption, the map $R \to \varphi_*R$ has an $R$-module retraction. Identifying $\varphi_*R$ with $R$, we have additive maps $\varphi\!:R \to R$ and $\psi:R \to R$ such that $\psi \circ \varphi =$ id$_R$. By functoriality, the induced map $\Phi\!:\H^i_\m(R) \to \H^i_\m(R)$ also splits, via a map $\Psi\!:\H^i_\m(R) \to \H^i_\m(R)$ induced by $\psi$. One can check by direct computations that $\Psi(r\Phi(\eta)) = \psi(r) \eta$ for all $r \in R$ and $\eta \in \H^i_\m(R)$. We then have that 
\[
\ds y = \Psi^{e'}(\Phi^{e'}(y)) =\sum_{i=1}^n \Psi^{e'}(r_i\Phi^{e'+i}(y)) = \sum_{i=1}^n \psi^{e'}(r_i) \Phi^i(y) \in W.
\]
\end{proof}

In the rest of the section, assume that $S$ is an $F$-finite complete regular local ring, and we let $I \subseteq S$ be an ideal. Assume that $R=S/I$ is $F$-pure, and let $\m$ denote its maximal ideal. Since $S/I$ is $F$-pure, we have that $\Ext^i_S(S/I,S)$ is a root of $\H^i_I(S)$, that is, the generating morphism $\Ext^i_S(S/I,S) \to \Ext^i_S(S/I^{[p]},S)$ is an injection \cite[Theorem 2.8]{AnuragUliPure}. Let $\ck{(-)}$ denote the Matlis duality functor. For any positive integer $e$, given a cofinite $R\lrphi{e}$-module $M$, we have a filtration 
\begin{equation} \label{Lyufilt}
\ds 0 = L_0 \subseteq N_0 \subseteq L_1 \subseteq \ldots \subseteq L_t \subseteq N_t = M
\end{equation}
of $R\lrphi{e}$-modules such that, for all $0 \leq j \leq t$, $L_{j+1}/N_j$ is a non-zero simple $R\lrphi{e}$-module on which $\varphi_e$ acts injectively. In fact, this comes from a direct generalization of \cite[Theorem 4.7]{LyuFmod} to $e$-th Frobenius actions. Furthermore, for all $j$, the $R\lrphi{e}$-module $N_j/L_j$ is $\varphi_e$-nilpotent, that is, $\varphi_e^i (N_j/L_j) = 0$ for some $i >0$. Finally, the isomorphism classes of the modules $L_j/N_{j-1}$ are invariants of $M$ \cite[Theorem 4.7]{LyuFmod}.

\begin{remark} \label{Remfiltration_Fpure}
Let $\varphi_e\!:R \to R$ be an $e$-th Frobenius action, and assume that the map $R \to \varphi_{e*} R$ splits. It follows from Lemma \ref{generalization Linquan} that the filtration (\ref{Lyufilt}) reduces to a filtration
\begin{equation} \label{filtration_Fpure_gen} 
\ds 0 = L_0 \subsetneq L_1 \subsetneq \ldots \subsetneq L_t = \H^i_\m(R)
\end{equation}
of $R\lrphi{e}$-modules such that $L_{j+1}/L_j$ is a non-zero simple $R\lrphi{e}$-module, on which $\varphi_e$ acts injectively. In fact, in the filtration (\ref{Lyufilt}), we must have $L_j=N_j$ for all $j$, otherwise $\varphi_e$ would not act injectively on $\H^i_\m(R)/L_j$, contradicting the anti-nilpotency of $\H^i_\m(R)$.
\end{remark}

Let $(S,\n,\k)$ be a complete $F$-finite regular local ring of dimension $n$, and let $I \subseteq S$ be such that $R=S/I$ is $F$-pure. Let $\m = \n/I$ denote the maximal ideal of $R$. Assume that $\H^i_\m(R) \ne 0$, and let $F\!:R \to R$ denote the standard Frobenius action on $R$. For any $f \notin \cP(R)$, we have by definition that the inclusion $F^e_*f \cdot R \subseteq F^e_*R$ splits for some $e \gs 1$. This map induces an $e$-th Frobenius action $fF^e\!:\H^i_\m(R) \to \H^i_\m(R)$. By a repeated use of Lemma \ref{generalization Linquan} and Remark \ref{Remfiltration_Fpure}, we obtain a filtration 
\begin{equation} \label{filtration_Fpure} 
\ds 0 = L_0 \subsetneq L_1 \subsetneq \ldots \subsetneq L_t = \H^i_\m(R)
\end{equation}
of $R\lrphi{e}$-modules such that $L_{j+1}/L_j$ is a non-zero simple $R\lrphi{e}$-module, on which $\varphi_e$ acts injectively. An equivalent way of saying this is to say that the composition
\[
\xymatrixcolsep{5mm}
\xymatrixrowsep{2mm}
\xymatrix{
\ds L_{j+1}/L_j \ar[rr] &&  F^e_*\left(L_{j+1}/L_{j}\right) \ar[rr]^-{\cdot F^e_*f} &&  F^e_*\left(L_{j+1}/L_{j}\right)
}
\]
is injective. 
 
If we apply the Matlis duality functor to the filtration (\ref{filtration_Fpure}), we obtain a chain of $S$-modules
\[
\ds 0 = M_0 \subsetneq M_1 \subsetneq \ldots \subsetneq M_t = \Ext_S^{n-i}(R,S),
\]
which are roots for the following filtration of $F^e_S$-modules
\[
\ds 0 = \mathcal{M}_0 \subsetneq \mathcal{M}_1 \subsetneq \ldots \subsetneq \mathcal{M}_t = \H^{n-i}_I(S).
\]
Note that $M_j \cong \ck{\left(L_t/L_{t-j}\right)}$ for all $j=0,\ldots,t$, so that $M_j/M_{j-1} \cong \ck{\left(L_{t-j+1}/L_{t-j}\right)}$. In addition, the quotient $M_j/M_{j-1}$ is a root of the non-zero simple $F^e_S$-module $\mathcal{M}_j/\mathcal{M}_{j-1}$. Since it is a simple root, $M_j/M_{j-1}$ can have only one associated prime. In fact, if $P \in \Ass_S(M_j/M_{j-1})$ was not the only associated prime of such a module, then $\H^0_P(M_j/M_{j-1})$ would be the root of a non-zero proper $F^e_S$-submodule of $\mathcal{M}_j/\mathcal{M}_{j-1}$, giving a contradiction. In addition, since $M_j \subseteq \Ext_S^{n-i}(R,S)$, we have that $M_j$ is an $R$-module, since $IM_j=0$.
 
In particular, if $P \in \Ass_S(M_j/M_{j-1})$, we have that $I \subseteq P$.
Let 
\begin{equation}\label{EqLambda}
\ds \Lambda^i_{R} = \left\{ \p = P/I\in \Spec(R) \mid \{P\} = \Ass_S(M_j/M_{j-1}) \mbox{ for some } j=1,\ldots,t\right\}.
\end{equation}
By our observations, we note that $\Lambda_R^i$ only depends on $R$, and not on the presentation $S/I$. Furthermore, since the isomorphism classes of the modules $L_{t-j+1}/L_{t-j}$ are invariants of $\H^i_\m(R)$, the set $\Lambda^i_R$ is independent of the chosen filtration of $\H^i_\m(R)$.

\begin{lemma} \label{containment_in_cP}
Let $(S,\n,\k)$ be an $F$-finite regular local ring. Let $I \subseteq S$ be an ideal such that $R=S/I$ is $F$-pure. Then $\p \subseteq \cP(R)$ for all $\p \in \Lambda^i_R$.
\end{lemma}
\begin{proof}
First, assume that $S$ is complete. Let $P,Q \in \Spec(S)$ be such that $P/I = \p$, and $Q/I = \cP(R)$. Then $P$ is the unique associated prime of $M_j/M_{j-1}$, for some $j=1,\ldots,t$. Denote $N:=M_j/M_{j-1}$. Since $N$ is the Matlis dual of an Artinian module, it is finitely generated. Let $g \in P$, and assume that $g \notin Q$. Since $P$ is the only associated prime of $N$, we have that $\H^0_P(N) = N$. As a consequence, we can find $s \gg 0$ such that $P^s N = 0$. Let $f = g^s$, so that $f \notin Q$ and $f \cdot N = f \cdot \H^0_P(N) = 0$. Let $e \gg 0$ be an integer such that $R \to F^e_*R \stackrel{\cdot F^e_*f}{\longrightarrow} F^e_*R$ splits, and consider the induced Frobenius action on $\H^i_\m(R)$. By Lemma \ref{generalization Linquan}, and the subsequent discussion, the map
\[
\xymatrixcolsep{5mm}
\xymatrixrowsep{2mm}
\xymatrix{
\ds L_{t-j}/L_{t-j-1} \ar[rr] &&  F^e_*\left(L_{t-j}/L_{t-j-1}\right) \ar[rr]^-{\cdot F^e_*f} &&  F^e_*\left(L_{t-j}/L_{t-j-1}\right)
}
\]
is injective. Applying Matlis duality, we obtain that
\[
\xymatrixcolsep{5mm}
\xymatrixrowsep{2mm}
\xymatrix{
\ds F^e_*N  \ar[rr]^-{\cdot F^e_*f} && F^e_*N \ar[rr] && N
}
\]
is surjective. However, we are assuming that $F^e_*f \cdot F^e_*N = F^e_*(f \cdot N) = 0$, and the surjection implies that $N=0$. This is a contradiction, and concludes the proof in the complete case. 

In the general case, note that $P \widehat{S}$ is radical, because $S$ is excellent. Write $P\widehat{S} = P_1 \cap \ldots \cap P_\ell$ for some $P_1,\ldots,P_\ell$ prime ideals in $\widehat{S}$. The module $M_j/M_{j-1} \otimes_S \widehat{S} \cong \widehat{M_j/M_{j-1}}$ may not be the root of a simple $F_{\widehat{S}}$-module, but we can extend the filtration 
\[
\ds 0 \subseteq \widehat{M_0} \subseteq \ldots \subseteq \widehat{M_{j-1}} \subseteq \widehat{M_j} \subseteq \ldots \subseteq \widehat{\Ext^i_S(R,S)} \cong \Ext_{\widehat{S}}^i(\widehat{R},\widehat{S})
\]
to
\[
\ds 0 \subseteq \widehat{M_0} \subseteq \ldots \widehat{M_{j-1}} = M_{j,1} \subseteq \ldots \subseteq M_{j,i_j} = \widehat{M_j} \subseteq \ldots \subseteq \Ext_{\widehat{S}}^i(\widehat{R},\widehat{S}),
\]
in a way that $M_{j,r}/M_{j,r-1}$ is a root of a simple $F_{\widehat{S}}$-module. The primes $P_1,\ldots,P_\ell$ appear as the unique associated primes of some of such factors. By the complete case, we have that $P_1,\ldots,P_\ell \subseteq Q'$, where $Q'$ is the lift to $\widehat{S}$ of the splitting prime $\cP(\widehat{R})$. Since $\cP(\widehat{R}) = \cP(R) \widehat{R}$ \cite[Proposition 3.9]{AE}, we obtain that $Q' = Q \widehat{S} = \widehat{Q}$, and thus $P \widehat{S}  = P_1 \cap \ldots \cap P_\ell \subseteq Q \widehat{S}$. Contracting back to $S$, we finally conclude that $P \subseteq Q$.
\end{proof}
We now recall the definition of a compatible ideal, which was introduced by Schwede \cite{KarlCentersFpurity}.

\begin{definition} Let $R$ be an $F$-finite ring of prime characteristic. An ideal $I \subseteq R$ is said to be compatible if $\varphi(F^e_*I) \subseteq I$ for all $e>0$, and all $\varphi \in \Hom_R(F^e_*R,R)$.
\end{definition}

An equivalent way of formulating the compatibility of an ideal is in terms of the Cartier algebra of $R$. We now recall its definition. For an integer $e>0$ let 
\[
\ds \cC_e(R) = \{\phi \in \Hom_\ZZ(R,R) \mid r\phi(s) = \phi(r^{p^e}s)\mbox{ for all } r,s \in R\}.
\]
The Cartier algebra of $R$ is the graded $\FF_p$-algebra $\cC(R) = \bigoplus \cC_e(R)$. An ideal $I \subseteq R$ is then compatible if and only if $\phi(I) \subseteq I$ for all $\phi \in \cC_e(R)$, for all positive integers $e$. For more details about this approach we refer to \cite{BB-CartierMod}. Examples of compatible ideals include the splitting prime $\cP(R)$ and the test ideal $\tau(R)$. In addition, every associated prime of a compatible ideal is again compatible \cite[Corollary 4.8]{KarlCentersFpurity}.

\begin{thm} \label{compatibility filtration} Let $(S,\n,\k)$ be an $F$-finite regular local ring. Let $I \subseteq S$ be an ideal such that $R=S/I$ is $F$-pure. Then every $\p \in \Lambda^i_R$ is a compatible ideal.
\end{thm}
\begin{proof}
Let $P$ be a lift of $\p$ to $S$, so that $P$ is the unique associated prime of some $M_j/M_{j-1}$ appearing in the filtration. We localize at $P$, and obtain a new filtration
\[
\ds 0 \subseteq (M_0)_P \subseteq \ldots \subseteq (M_t)_P \subseteq \left(\Ext_S^i(R,S)\right)_P \cong \Ext^i_{S_P}(R_\p,S_P).
\]
We remove from the filtration the factors that become zero after localization. Note that the prime $PS_P$ is associated to the factor $(M_j)_P/(M_{j-1})_P$. By Lemma \ref{containment_in_cP}, we have that $PS_P/IS_P \cong \p R_\p \subseteq \cP(R_\p) \subseteq \p R_\p$, which forces $\p R_\p= \cP(R_\p)$. In particular, $\p R_\p$ is a compatible ideal of $R_\p$. Let $e>0$, let $\varphi \in \Hom_R(F^e_*R,R)$ and assume, by way of contradiction, that $\varphi(F^e_*\p) \not\subseteq\p$. Let $\varphi_\p$ denote the map induced by $\varphi$ on the localizations at $\p$, so that $\varphi_\p \in \Hom_{R_\p}(F^e_*R_\p,R_\p)$. Then $\varphi(F^e_*\p)R_\p = \varphi_\p\left(F^e_*(\p R_\p)\right)  = R_\p$, contradicting the assumption that $\p R_\p$ is compatible.
\end{proof}

We now present two results that follow from Theorem \ref{compatibility filtration}. The first one is the aforementioned generalization of the fact that every minimal prime over $\Supp(\Ext^i_S(R,S))$ is compatible. 
\begin{corollary} \label{AssPrimesLC} Let $(S,\n,\k)$ be an $F$-finite regular local ring, and $I \subseteq S$ be an ideal such that $R=S/I$ is $F$-pure. For any integer $i \in \NN$ such that $\H^i_I(S) \ne 0$, and for any $P \in \Ass_S(\H^i_I(S))$, the ideal $P/I \in \Spec(R)$ is compatible.
\end{corollary}

\begin{proof}
This follows immediately from Theorem \ref{compatibility filtration}, and the fact that $P/I \in \Lambda^{n-i}_R$ for any $P \in \Ass_S(\H^i_I(S))$.
\end{proof}

We conclude this section by showing that the ideal defining the non-Cohen-Macaulay locus of an $F$-pure equidimensional ring is compatible. Recall that, if $R$ is an $F$-finite local ring, then it admits a dualizing complex \cite[Remark 13.6]{Gabber}, and hence the non-Cohen-Macaulay locus $\nCM(R) = \{\p \in \Spec(R) \mid R_\p$ is not Cohen-Macaulay$\}$ is closed. In particular, if $R=S/I$, with $(S,\n,\k)$ a regular local ring, then $\nCM(R) = V(\a)$ is closed. In this case, further assuming that $R$ is equidimensional, it follows from local duality that
\[
\ds \a=\bigcap_{i \ne \Ht(I)} \ann_S(\Ext^i_S(R,S)).
\]

Theorem \ref{compatibility filtration} now gives that the ideal defining the non-Cohen-Macaulay locus of an equidimensional $F$-pure ring is compatible.

\begin{corollary} \label{nCMidealComp} Let $(S,\n,\k)$ be an $F$-finite regular local ring, and $I \subseteq S$ be an ideal such that $R=S/I$ is $F$-pure. Assume that $R$ is equidimensional, and let nCM$(R) = V(\a)$. Then $\sqrt{\a}$ is a compatible ideal of $R$. In particular, $R$ is Cohen-Macaulay if and only if $R_{\cP(R)}$ is Cohen-Macaulay.
\end{corollary}
\begin{proof}
We may assume that $\a$ is radical. Every prime $\p$ that is minimal over $\a$ will be minimal over $\Ext^i_S(R,S)$ for some $i$. Since associated primes of such a module always appear in $\Lambda^i_R$, Theorem \ref{compatibility filtration} yields that $\p$ is compatible. Finally, intersection of compatible ideals is compatible \cite[lemma 3.5]{KarlCentersFpurity}, hence $\a$ is compatible. The last statement follows from the fact that $\cP(R)$ is the largest compatible ideal of $R$.
\end{proof}

\begin{comment}
\end{comment}

%%%%%%%%%%%%%%%%%%%%%%%%%%%%%%%%%%%%%%%%%%%%%%%%%%%%%%%%%
\section{A formula for Lyubeznik numbers of graded $F$-pure rings}
%%%%%%%%%%%%%%%%%%%%%%%%%%%%%%%%%%%%%%%%%%%%%%%%%%%%%%%%%

In this section, we consider standard graded algebras over a field $\k$ of positive characteristic $p>0$. By standard graded $\k$ algebra, we mean a graded algebra of finite type over $\k$, that is generated by elements of degree one. The typical example is the polynomial ring $S=\k[x_1,\ldots,x_n]$ in finitely many variables over $\k$, with $\deg(x_i)=1$ for all $i$. Moreover, any standard graded $\k$-algebra can be written as a quotient of a polynomial ring of this form by a homogeneous ideal. The condition that a standard graded $\k$ algebra is $F$-finite is equivalent to the requirement that $[\k^{1/p}:\k]<\infty$. In particular, standard graded $\k$-algebras over a perfect field are $F$-finite \cite[Lemma 1.5]{FedderFputityFsing}.

%%%%%%%%%%%%%%%%%%%%%%%%%%%%%%%%%%%%%%%%%%%%%%%%%%%%%%%%%
\subsection{The formula}
%%%%%%%%%%%%%%%%%%%%%%%%%%%%%%%%%%%%%%%%%%%%%%%%%%%%%%%%%

Given a standard graded ring $(R,\m,\k)$ and a graded $R$-module $M$, $\H^i_I(M)$ is also graded. The graded Matlis dual of a graded $R$-module $M$ can be described as the graded $R$-module $\ck{M}$, with $j$-th graded piece given by $\ck{M}_j := \Hom_{\k} \left( M_{-j}, \k \right)$ \cite[Section 3.6]{BrHe}. 
For a standard graded polynomial ring $S=k[x_1,\ldots,x_n]$, with $\m=(x_1,\ldots,x_n)$, the graded canonical module $\omega_S$ is isomorphic to $\ck{\H^n_\m(S)} \cong S(-n)$.

\begin{thm}[{Graded local duality, see  \cite[Theorem 3.6.19]{BrHe}}]
Let $S=\k\left[ x_1, \ldots, x_n \right]$ be a polynomial ring over a field $\k$, with $\deg(x_i) = 1$, and $\m = (x_1,\ldots,x_n)$. For any finitely generated graded $S$-module $M$, there is a degree-preserving isomorphism
$$\H^i_{\m}(M) \cong \ck{\left(\Ext^{n-i} \left(M, S(-n) \right) \right)}.$$
\end{thm}

\begin{thm}[\cite{WZ-projective}]\label{ThmWZ}
Let $\k$ be a perfect field of characteristic $p>0$, $S=\k\left[ x_1, \ldots, x_n \right]$,  $I$ an homogeneous  ideal in $S$. 
For any integer $e \geq 1$, the $e$-th iteration of the Frobenius map $F^e\!:R\to R$ induces a Frobenius action on 
$$\left(  \Ext^{n-i}_S \left(  \Ext^{n-j}_S \left( R, \Omega_S \right), \Omega_S \right) \right)_0$$
where $\Omega_S \cong S(-n)$ is a graded canonical module of $S$, and
$$\lambda_{i,j}(R) = \dim_\k\left( \bigcap_{e\in\NN} F^e \left(  \Ext^{n-i}_S \left(  \Ext^{n-j}_S \left( R, \Omega_S \right), \Omega_S \right) \right)_0\right).$$
\end{thm}

\begin{remark}

Since $\Omega_S \cong S(-n)$,  there is a graded isomorphism
	$$ \Ext^{n-i}_S \left(  \Ext^{n-j}_S \left( R, \Omega_S \right), \Omega_S \right)  \cong  \Ext^{n-i}_S \left(  \Ext^{n-j}_S \left( R, S \right), S \right).
	$$
In particular, we obtain that
	$$\lambda_{i,j}(R) = \dim_\k\left( \bigcap_{e\in\NN} F^e \left(  \Ext^{n-i}_S \left(  \Ext^{n-j}_S \left( R, S \right), S \right) \right)_0\right).$$
\end{remark}

\begin{remark}\label{RemFrobInjSurj}
Let $\k$ be a perfect field of characteristic $p>0$, and let $V$ be a finite dimensional $\k$-vector space   with a Frobenius action.  
We note that $F_* \k\cong \k,$ because $\k$ is a perfect field.
Since $V=\oplus \k,$ we deduce that $F_* V\cong \oplus F_* \k \cong V$.
If $F:V\to V$ is injective, then the induced map $V\to F_* V$ is an injective $\k$-linear map. Since $V\cong F_* V$,  $V\to F_* V$ is also surjective. Then, $F:V\to V$ is an isomorphism. 
\end{remark}

\begin{thm}\label{lyu formula}
Let $\k$ be a field of characteristic $p>0$, $S=\k\left[ x_1, \ldots, x_n \right]$, and $I$ a homogeneous ideal in $S$ with $R = S/I$ $F$-pure. Then
$$\lambda_{i,j}(R) = \dim_\k \left( \Ext^{n-i}_S \left(  \Ext^{n-j}_S \left( R, S \right), S \right) \right)_0 .$$
\end{thm}
\begin{proof}
Let $\overline{\k}$ denote the algebraic closure of $\k.$
We note that if $R$ is $F$-pure, then $R\otimes_\k \overline{\k}$ is also $F$-pure.
We also observe that 
$\lambda_{i,j}(R) = \lambda_{i,j}(R\otimes_\k \overline{\k})$
and 
$$
\dim_\k \left( \Ext^{n-i}_S \left(  \Ext^{n-j}_S \left( R, S \right), S \right) \right)_0  =
\dim_{\overline{\k}} \left( \Ext^{n-i}_{S\otimes_\k \overline{\k}} \left(  \Ext^{n-j}_{S\otimes_\k \overline{\k}} \left( R
\otimes_\k \overline{\k}, S \otimes_\k \overline{\k}
 \right), S\otimes_\k \overline{\k} \right) \right)_0.
$$
Then, we may assume without loss of generality that $\k$ is algebraically closed.
We set 
$$N=\left( \Ext^{n-i}_S \left(  \Ext^{n-j}_S \left( F_*R, S \right), S \right) \right)_0.$$
We note that, since $\Ext^i_S(-, S)$ commutes with the action of Frobenius,
$$
\left( \Ext^{n-i}_S \left(  \Ext^{n-j}_S \left( F_*R, S \right), S \right) \right)_0\cong F_* N
$$
The Frobenius action on $N$ is induced by the inclusion map $R\to F_* R.$ Then, this must correspond to $N\to F_* N,$ which splits as   $R\to F_* R$ does.
Then,  $N\to F_* N$ is injective. As a consequence, the Frobenius 
map $F:N\to N$ is injective. Then, by Remark \ref{RemFrobInjSurj}, $F:N\to N$ is also surjective, and it follows from Theorem \ref{ThmWZ} that
$$
\lambda_{i,j}(R) = \dim_\k \left( \bigcap_{e\in\NN} F^eN\right) =\dim_\k(N).
$$
\end{proof}

%%%%%%%%%%%%%%%%%%%%%%%%%%%%%%%%%%%%%%%%%%%%%%%%%%%%%%%%%
\subsection{Examples of Lyubeznik numbers}\label{SubSecExamples}
%%%%%%%%%%%%%%%%%%%%%%%%%%%%%%%%%%%%%%%%%%%%%%%%%%%%%%%%%
In this subsection we use Theorem \ref{lyu formula} to compute the Lyubeznik numbers of various $F$-pure graded rings. As our formula was already known for Stanley-Reisner rings, we focus on other graded rings which are also $F$-pure.

\begin{notation}
Given a $d$-dimensional local ring  $(R,\m,\k)$ containing a field, or a standard graded $\k$-algebra, we collect all its Lyubeznik numbers in a table $\Lambda(R)=(\lambda_{i,j}(R))_{0 \leq i,j \leq d}$. 
\end{notation}
Suppose that $d=\dim(R).$
We recall that  $\lambda_{i,j}(R)=0$ for $i<j$, $i>d$ and $j>d$ \cite[4.4]{LyuDMod}. Then, the Lyubeznik Table is upper triangular. In addition, $\lambda_{d,d}(R)\neq 0$.  

\begin{notation}
Let   $(R,\m,\k)$ be $d$-dimensional local ring  containing a field.  Then,
\begin{itemize}
\item $\lambda_{d,d}(R)$ is called the highest Lyubeznik number of $R$. 
\item We say that $\Lambda(R)$ is trivial if $\lambda_{d,d}(R)=1$ and all the other Lyubeznik numbers vanish.
\end{itemize}
\end{notation}

It is well-known that Cohen-Macaulay rings in prime characteristic always have a trivial Lyubeznik table. In fact, this property remains true for sequentially Cohen-Macaulay rings. For this reason, we focus on rings that are not Cohen-Macaulay. Since there are already formulas and algorithms to compute Lyubeznik numbers for Stanley-Reisner rings \cite{YStr,Montaner,AMV}, we do not provide examples defined by monomials ideals.

We first include rings defined by binomial edge ideals. Given a simple graph $G$ with vertex set $\{1,\ldots,n\},$ the binomial edge ideal, $J_G\subseteq S=K[x_1,\ldots,x_n,y_1,\ldots,y_n]$,
associated to $G$ is defined by
$$
J_G= \left( x_iy_j-x_jy_i \;| \;  \{i,j\}\in G\right)\subseteq  S.
$$
For a study of $F$-purity of binomial edge ideals, we refer to the work of  Matsuda \cite{Matsuda}.

\begin{example}
The binomial edge ideal of the $5$-cycle defines an $F$-pure ring in characteristics $3$, $5$ and $7$ \cite[Example 2.7]{Matsuda}. These are not Cohen-Macaulay, because its vertex connectivity is equal to $2$ \cite{BNB17}. However, the Lyubeznik table is trivial.
\end{example}

We now compute a few example defined by binomial ideals via a semigroup. By the work of Hochster \cite{MelMonomials}, normal affine  semigroup  algebras are Cohen-Macaulay. For this reason, we discuss here only weakly normal semigroup algebras, which are still $F$-pure \cite[Proposition 5.3]{Yasuda}.

\begin{example}

Take the complete bipartite graph in $3$ and $5$ vertices, $K_{3,5}$. The following is the Lyubeznik table of its binomial edge ideal in characteristic $101$:
$$ \Lambda(S/J_{K_{3,5}}) = \bgroup\makeatletter\c@MaxMatrixCols=11\makeatother\begin{pmatrix}0&
      0&
      0&
      0&
      0&
      0&
      0&
      0&
      0&
      0&
      0\\
      &
      0&
      0&
      0&
      0&
      0&
      0&
      0&
      0&
      0&
      0\\
      &
      &
      0&
      0&
      0&
      0&
      0&
      0&
      0&
      0&
      0\\
      &
      &
      &
      0&
      0&
      0&
      0&
      0&
      0&
      0&
      0\\
      &
      &
      &
      &
      0&
      1&
      0&
      0&
      0&
      0&
      0\\
      &
      &
      &
      &
      &
      0&
      0&
      0&
      0&
      0&
      0\\
      &
      &
      &
      &
      &
      &
      1&
      1&
      0&
      0&
      0\\
      &
      &
      &
      &
      &
      &
      &
      0&
      0&
      0&
      0\\
      &
      &
      &
      &
      &
      &
      &
      &
      0&
      0&
      0\\
      &
      &
      &
      &
      &
      &
      &
      &
      &
      1&
      0\\
      &
      &
      &
      &
      &
      &
      &
      &
      &
      &
      1\\
      \end{pmatrix}\egroup.$$
\end{example}

\begin{example}[Bertin's example]\label{Bertin}
	Let $R = \mathbb{F}_2[x_1, x_2, x_3, x_4]$, and let $R^G$ denote the ring of invariants of the cyclic group $G = \ZZ/4$ on the variables of $R$. This is a famous example of Marie-Jos\'{e} Bertin \cite{Bertin} of a ring of invariants that is not Cohen-Macaulay, and it can be written as a quotient of $\mathbb{F}_2[y_1, \ldots, y_8]$. Macaulay2 computations show that $\lambda_{1,3} \left( R^G \right) = 1$, so the Lyubeznik table is non-trivial.
		
\end{example}

We now provide an example to show that $F$-purity is needed for the formula of Theorem~\ref{lyu formula}.
\begin{example}[{See  \cite[Example 2.3]{AnuragUliSegre}}]\label{AnuragUli}
	Let $R = \k[u,v,w,x,y,z]$ and consider the ideal
\[
I = (u^3+v^3+w^3, u^2x+v^2y+w^2z, x^2u+y^2v+z^2w,
    x^3+y^3+z^3, vz-wy, wx-uz, uy-vx).
    \]
It can be checked with Macaulay2 that $R/I$ is not $F$-pure, and that 
\[
\dim_\k\left(\Ext^6_R(\Ext^4_R(R/I,R),R)_0\right) = 1.
\]
However, $\lambda_{0,2}(R/I)=0$, since if $H^4_I(R)=0$ \cite[Example 2.3]{AnuragUliSegre}.
    \begin{comment}The Lyubeznik table of $R/I$ in characteristic $5$ is
    $$\Lambda(R/I) = \begin{pmatrix}0&
      0&
      1&
      0\\
      &
      0&
      0&
      0\\
      &
      &
      0&
      1\\
      &
      &
      &
      1\\
      \end{pmatrix}$$
      \end{comment}
\end{example}

%%%%%%%%%%%%%%%%%%%%%%%%%%%%%%%%%%%%%%%%%%%%%%%%%%%%%%%%
\section{Lyubeznik numbers of projective varieties}
%%%%%%%%%%%%%%%%%%%%%%%%%%%%%%%%%%%%%%%%%%%%%%%%%%%%%%%%%
Let $X$ be a projective variety over a field $\k$ of positive characteristic $p$. Then $X$ can be embedded in a projective space $\PP^n_\k=\Proj(S)$, where $S=\k[x_0,\ldots,x_n]$. Under this embedding, $X$ can be represented as $\Proj(R)$, where $R=S/I$ for some homogeneous ideal $I$.
The Lyubeznik number of $X$ with respect to $i,j$ is defined as
$$\lambda_{i,j} (X) := \dim_{\k} \Ext^i_S \left(\k, \H^{n-j}_I (S) \right).$$
This definition depends only on $X$, $i$ and $j$, meaning that it does not depend on the embedding of $X$ in $\PP^n_\k$ \cite{WZ-projective}.
In this section, we seek to relate the numbers $\lambda_{i,j} (X) $ with the geometry of $X$.

%%%%%%%%%%%%%%%%%%%%%%%%%%%%%%%%%%%%%%%%%%%%%%%%%%%%%%%%%
\subsection{Lyubeznik numbers that measure connectedness}
%%%%%%%%%%%%%%%%%%%%%%%%%%%%%%%%%%%%%%%%%%%%%%%%%%%%%%%%%
We start our study of $\lambda_{i,j} (X) $ focusing on $\lambda_{0,1} (X)$. The Hartshorne-Lichtenbaum Vanishing Theorem \cite{HartCD}, and the Second Vanishing Theorem for local cohomology over a local ring \cite{HartCD,O,P-S,H-L} imply that this number relates to the number of connected components of the punctured spectrum. As a natural extension, we verify a similar relation for projective varieties. We point out that  Theorem \ref{ThmHighest} may be already known to experts. However, to the best of our knowledge, this result has not been recorded explicitly in the literature.
We start with a key lemma.

\begin{lemma}\label{LemmaNumberComp}
Let $X$ be a projective variety over an algebraically closed field $\k$ of characteristic $p>0$. Suppose that $X$ has $t$ connected components, which are all of dimension at least $1$. Then, $\lambda_{0,1}(X)=t-1$.
\end{lemma}
\begin{proof}
Let $S=\k[x_0,\ldots,x_n]$ and $I \subseteq \m=(x_0,\ldots,x_n)$, so that $X=\Proj(S/I)$. Fix ideals $J_1,\ldots,J_t$ of $S$ for which $\cV(J_1), \ldots, \cV(J_t)$ are precisely the connected components of $X$.
We proceed by induction on $t$.
If $t=1,$ then $\H^n_{J_1}(S)=0$ by the Second Vanishing Theorem for local cohomology.
Then, $\lambda_{0,1}(X)=\dim_\k\Hom_S(\k,\H^n_{J_t}(S))=0.$

We assume the claim for $t-1,$ and prove it for $t$.
Let $J=J_1\cap\ldots\cap J_{t-1}.$
The Mayer-Vietoris sequence associated to $J$ and $J_t$ is
\[
\cdots \to \H^{n}_{J+J_t}(S)\to \H^{n}_{J}(S)\oplus \H^{n}_{J_t}(S)\to \H^{n}_{I}(S)\to \H^{n+1}_{J+J_t}(S)\to 0.
\]
Since the dimension of every irreducible component of $X$ is at least one, we have that every minimal prime of $I$ has dimension at least two, and $\H^n_I(S),\H^n_{J_t}(S),$ and $\H^{n}_{J}(S)$ have zero-dimensional supports as a consequence of the Hartshorne-Lichtenbaum Vanishing Theorem \cite{HartCD}. Then, $\H^n_I(S),\H^n_{J_t}(S),$ and $\H^{n}_{J}(S)$ are a finite direct sum of the injective  hull of the residue filed $E_S(\k).$
By the Second Vanishing Theorem for local cohomology, we have that $\H^{n}_{J_t}(S)=0.$
In addition, $\H^{n-1}_J(S)=E_S(\k)^{\oplus t-2}$ by induction hypothesis.
Furthermore, $\sqrt{J+J_t}=\m$ by construction of $J_i.$
Then the previous long exact sequence becomes
\[
0 \to E_S(\k)^{\oplus t-2}\to \H^{n-1}_{I}(S)\to E_S(\k)\to 0.
\]
Thus, $\H^{n-1}_I(S)=E_S(\k)^{\oplus t-1}$.
Hence, $\lambda_{0,1}(X)=\dim_\k\Hom_S(\k,\H^n_I(S))=t-1.$
\end{proof}

We now extend the previous result to any projective variety over a field of characteristic $p>0$.
\begin{thm}\label{ThmHighest}
Let $X$ be a projective variety over a field $\k$ of characteristic $p>0$. Let $Y$ denote the union of the irreducible components of $X\times_\k \overline{\k}$ of dimension at least $1$. Let $t$ denote the number of connected components of $Y$.
Then, $\lambda_{0,1}(X)=t-1.$
\end{thm}
\begin{proof}
Given Lemma \ref{LemmaNumberComp}, we may assume that $X\neq Y$. Let $S=\k[x_1,\ldots,x_n]$ and $I \subseteq \m=(x_0,\ldots,x_n)$ be such that $X=\Proj(S/I)$. Let $\overline{X}=X\times_\k \overline{\k}.$ Then, $\overline{X}=\Proj(S/I\otimes_\k \overline{\k})$. 
Lyubeznik numbers are not affected by field extensions, hence we may assume with no loss of generality that $\k=\overline{\k}$. Let $\p_1,\ldots,\p_s$ be the minimal primes of $I$ of dimension at least $2$, and $\q_1,\ldots,\q_r$ be the minimal primes of $I$ of dimension $1$. Let $J_1=\q_1\cap \ldots\cap \q_r$ and $J_2=\p_1\cap\ldots\cap\p_s$. Then, $I=J_1\cap J_2$. There exists a homogeneous polynomial  $f\in J_2\setminus \bigcup \{\p\in \Ass_S S/I\;|\; \dim R/\p=1\}$ by prime avoidance.

We claim that $\sqrt{I+fS}=J_2$. 
Let $Q$ be a minimal prime ideal of $\sqrt{I+fS}$. 
Since $\sqrt{I+fS}\subseteq J_2,$ we have $\dim R/\sqrt{I+fS}\gs 2$ and $Q\neq \m.$
We have that $I\subseteq Q.$ Then, $Q$ must contain a minimal prime of $I$. We note that $Q$ does not contain $\q_i;$ otherwise, $\m=\sqrt{\q_i+fS}\subseteq Q$. 

Let $T$ denote $\widehat{S_\m}.$
We have a long exact sequence given by
$$
\ldots\to \H^{n-1}_{J_2}(T)\to \H^{n-1}_I(T)\to \H^{n-1}_I(T_f)\to \H^{n}_{J_2}(T)\to \H^{n}_I(T)\to \H^{n}_I(T_f)\to 0,
$$
where the sequence stops at $n$ by the Hartshorne-Lichtenbaum Vanishing Theorem, because $\dim(T)=n+1$. Then, $T_f$ is a regular ring of dimension $n-1.$ We note that $\dim T_f/IT_f=\dim T_f/J_1T_f=0$. Then,  $\H^{n-1}_I(T_f)=0.$
We get a short exact sequence
$$
0\to \H^{n}_{J_2}(T)\to \H^{n}_I(T)\to \H^{n}_I(T_f)\to 0.
$$
This induces a long exact sequence: 
$$
0\to \Ext^0_S(\k,\H^{n}_{J_2}(T))\to\Ext^0_S(\k, \H^{n}_I(T))\to \Ext^0_S(\k,\H^{n}_I(T_f))\to \ldots.
$$
Since $\Ext^0_S(\k,\H^{n}_I(T_f))=0,$  we have that $\Ext^0_S(\k,\H^{n}_{J_2}(T))\cong\Ext^0_S(\k, \H^{n}_I(T))$.
Then, 
\begin{align*}
\lambda_{0,1}(X)
&=\dim_\k \Ext^0_S(\k,\H^{n}_{I}(S))\\
&=\dim_\k \Ext^0_S(\k,\H^{n}_{I}(T))\\
&=\dim_\k \Ext^0_S(\k, \H^{n}_{J_2}(T))\\
&=\dim_\k \Ext^0_S(\k, \H^{n}_{J_2}(S))\\
&= \lambda_{0,1}(Y)\hbox{ because }Y=\Proj(R/J_2);\\
&=t-1\hbox{ by Lemma \ref{LemmaNumberComp}}.
\end{align*}
\end{proof}

%%%%%%%%%%%%%%%%%%%%%%%%%%%%%%%%%%%%%%%%%%%%%%%%%%%%%%%%%
\subsection{Other Lyubeznik numbers}
%%%%%%%%%%%%%%%%%%%%%%%%%%%%%%%%%%%%%%%%%%%%%%%%%%%%%%%%%
In this subsection, we relate the first row in the Lyubeznik table of $X$ with the sheaf cohomology of the variety $X$. 

\begin{thm}\label{ThmSheaf}
Let $X$ be globally $F$-split, Cohen-Macaulay and equidimensional projective variety. Then
\[
\ds \lambda_{0,j}(X) = \dim_\k\left(\H^j(X,\OO_X)\right)
\]
for all $2 \ls j \ls n$.
\end{thm}
\begin{proof} 
Let $S=\k[x_0,\ldots,x_n]$ and $I \subseteq \m=(x_0,\ldots,x_n)$ be such that $X=\Proj(S/I)$. Let $2 \ls j \ls n$ be an integer, and let $R=S/I$. Since $R$ is $F$-pure, we have an inclusion $\ds \Ext^{n+1-j}_S(R,S) \hookrightarrow \H^{n+1-j}_I(S)$ \cite[Theorem 1.3]{AnuragUliPure} (see also \cite{DDSM}). After applying $\Hom_S(\k,-)$ and looking at $\k$-vector space dimensions, we obtain that 
\[
\ds \dim_\k\left(\Hom_S\left(\k,\Ext^{n+1-j}_S(R,S)\right)\right) \ls \dim_\k\left(\Hom_S\left(\k,\H^{n+1-j}_I(S)\right)\right) = \lambda_{0,j}(X).
\]
The reverse inequality already holds \cite[Theorem 2.1]{HunekeSharp}; therefore, we have equality. By graded local duality, we have that
\[
\ds \left(\Ext^{n+1-j}_S(R,S)\right)^{\vee} \cong  \H^j_\m(R)(n+1),
\]
and the latter is just a graded shift of $\H^j_\m(R)$. Since $R_\p$ is Cohen-Macaulay for every homogeneous prime $\p \ne \m$, and $R$ is $F$-pure, we have that $\H^j_\m(R) = \H^j_\m(R)_0$ \cite[Proposition 4.1]{Schenzel} (see also \cite[Corollary 1.3]{MaBuchsbaum}). In particular, $\H^j_\m(R)$ is a finitely generated $\k$-vector space. Therefore, $\Ext^{n+1-j}_S(R,S) \cong \ck{H^j_\m(R)}$ is also a finitely generated $\k$-vector space, and 
\[
\ds \dim_\k\left(\Hom_S(\k,\Ext^{n+1-j}_S(R,S))\right) = \dim_\k\left(\Ext^{n+1-j}_S(R,S)\right) = \dim_\k\left(\H^j_\m(R)\right).
\]
Putting everything together, we finally obtain that for $2 \ls j \ls n$
\[
\ds \lambda_{0,j}(X) = \dim_\k\left(\H^j_\m(R)_0\right) = \dim_\k\left(\H^j(X,\OO_X)\right).
\]
\end{proof}

In order to be able to obtain all the Lyubeznik numbers of $X$, we adapt to projective varieties a result known for local rings \cite{B-B,BlickleLyu}. We include a proof for the sake of completeness.
\begin{proposition}\label{PropDuality}
Let $X$ be a $d$-dimensional Cohen-Macaulay projective variety over a field $\k$ of characteristic $p$. Then, 
$$\lambda_{d+1,d+1}(X)=\lambda_{0,1}(X)+1$$
and $\lambda_{0,j}(X)=\lambda_{d+2-j,d+1}(X)$ for $2\ls j\ls d.$
\end{proposition}
\begin{proof}
Let $S=\k[x_1,\ldots,x_n]$ and $I \subseteq \m=(x_0,\ldots,x_n)$ be such that $X=\Proj(S/I)$. Let $R=S/I$. Since $X$ is Cohen-Macaulay, $R_Q$ is a Cohen-Macaulay ring for every homogeneous prime ideal $Q$ such that $Q\neq \m$. We note that $\H^j_I(S_Q)=0$ for $i\neq n-d-1$ \cite[Proposition 4.1]{P-S}.
Then, $\H^j_I(S_Q)=0$ is supported only at the maximal homogeneous ideal for $i\neq n-d-1,$ and so, $\H^i_\m \H^j_I(S)=0$ for $i>0.$
From Grothendieck spectral sequence $\H^i_\m \H^j_I(S) \Rightarrow \H^{i+j}_\m(S)$, we have that 
 $\H^{d+1}_\m \H^{n-d-1}_I(S)\cong \H^0_\m \H^{n-1}_I(S)\oplus E_S(\k)$
and $\H^0_\m \H^{n-d-2+j}_I(S)\cong \H^j_\m \H^{n-d-1}_I(S)$ for $2\ls j\ls d.$
The claim follows from the fact that
$\lambda_{i,j}(R)=\dim_\k \Hom_S(\k,\H^i_\m \H^{n-j}_I(S))$ \cite[Lemma 1.4]{LyuDMod}.
\end{proof}

\section*{Acknowledgments}

We thank Jack Jeffries for suggesting that we look at Example \ref{Bertin}. We also thank Moty Katzman for pointing out a mistake in Example \ref{AnuragUli} of a previous version of this article.

\bibliographystyle{alpha}
\bibliography{References}

\begin{thebibliography}{NBWZ13}

\bibitem[AE05]{AE}
Ian~M Aberbach and Florian Enescu.
\newblock The structure of {F}-pure rings.
\newblock {\em Mathematische Zeitschrift}, 250(4):791--806, 2005.

\bibitem[{\`A}M00]{Montaner}
Josep {\`A}lvarez~Montaner.
\newblock Characteristic cycles of local cohomology modules of monomial ideals.
\newblock {\em J. Pure Appl. Algebra}, 150(1):1--25, 2000.

\bibitem[{\`A}M04]{AM-Proc}
Josep {\`A}lvarez~Montaner.
\newblock Some numerical invariants of local rings.
\newblock {\em Proceedings of the American Mathematical Society},
  132(4):981--986, 2004.

\bibitem[AMV14]{AMV}
Josep \`Alvarez~Montaner and Alireza Vahidi.
\newblock Lyubeznik numbers of monomial ideals.
\newblock {\em Trans. Amer. Math. Soc.}, 366(4):1829--1855, 2014.

\bibitem[{\`A}MY]{AMYLyuNum}
Josep {\`A}lvarez~Montaner and Kohji Yanagawa.
\newblock Lyubeznik numbers of local rings and linear strands of graded ideals.
\newblock {\em To appear in Nagoya Math. J.}

\bibitem[BB05]{B-B}
Manuel Blickle and Raphael Bondu.
\newblock Local cohomology multiplicities in terms of \'etale cohomology.
\newblock {\em Ann. Inst. Fourier (Grenoble)}, 55(7):2239--2256, 2005.

\bibitem[BB11]{BB-CartierMod}
Manuel Blickle and Gebhard B{\"o}ckle.
\newblock Cartier modules: finiteness results.
\newblock {\em J. Reine Angew. Math.}, 661:85--123, 2011.

\bibitem[Ber67]{Bertin}
Marie-Jos\'{e} Bertin.
\newblock Anneaux d'invariants d'anneaux de polynomes, en caract\'{'e}ristique
  p.
\newblock {\em C. R. Acad. Sci. Paris}, (264), 1967.

\bibitem[BH93]{BrHe}
Winfried Bruns and J{\"u}rgen Herzog.
\newblock {\em Cohen-{M}acaulay rings}, volume~39 of {\em Cambridge Studies in
  Advanced Mathematics}.
\newblock Cambridge University Press, Cambridge, 1993.

\bibitem[Bli07]{BlickleLyu}
Manuel Blickle.
\newblock Lyubeznik's invariants for cohomologically isolated singularities.
\newblock {\em J. Algebra}, 308(1):118--123, 2007.

\bibitem[BNB17]{BNB17}
Arindam Banerjee and Luis N{\'u}{\~n}ez-Betancourt.
\newblock Graph connectivity and binomial edge ideals.
\newblock {\em Proc. Amer. Math. Soc.}, 145(2):487--499, 2017.

\bibitem[DDSM]{DDSM}
Hailong Dao, Alessandro De~{S}tefani, and Linquan Ma.
\newblock Cohomologically full rings.
\newblock {\em Preprint}.

\bibitem[Fed83]{FedderFputityFsing}
Richard Fedder.
\newblock {$F$}-purity and rational singularity.
\newblock {\em Trans. Amer. Math. Soc.}, 278(2):461--480, 1983.

\bibitem[Gab04]{Gabber}
Ofer Gabber.
\newblock Notes on some {$t$}-structures.
\newblock In {\em Geometric aspects of {D}work theory. {V}ol. {I}, {II}}, pages
  711--734. Walter de Gruyter, Berlin, 2004.

\bibitem[GLS98]{GarciaSabbah}
R.~Garc{\'{\i}}a~L{{\'o}}pez and C.~Sabbah.
\newblock Topological computation of local cohomology multiplicities.
\newblock {\em Collect. Math.}, 49(2-3):317--324, 1998.
\newblock Dedicated to the memory of Fernando Serrano.

\bibitem[Har68]{HartCD}
Robin Hartshorne.
\newblock Cohomological dimension of algebraic varieties.
\newblock {\em Ann. of Math. (2)}, 88:403--450, 1968.

\bibitem[HH94]{HoHuOmega}
Melvin Hochster and Craig Huneke.
\newblock Indecomposable canonical modules and connectedness.
\newblock In {\em Commutative algebra: syzygies, multiplicities, and birational
  algebra ({S}outh {H}adley, {MA}, 1992)}, volume 159 of {\em Contemp. Math.},
  pages 197--208. Amer. Math. Soc., Providence, RI, 1994.

\bibitem[HL90]{H-L}
Craig Huneke and Gennady Lyubeznik.
\newblock On the vanishing of local cohomology modules.
\newblock {\em Invent. Math.}, 102(1):73--93, 1990.

\bibitem[Hoc72]{MelMonomials}
Melvin Hochster.
\newblock Rings of invariants of tori, {C}ohen-{M}acaulay rings generated by
  monomials, and polytopes.
\newblock {\em Ann. of Math. (2)}, 96:318--337, 1972.

\bibitem[HR76]{HRFpurity}
Melvin Hochster and Joel~L. Roberts.
\newblock The purity of the {F}robenius and local cohomology.
\newblock {\em Advances in Math.}, 21(2):117--172, 1976.

\bibitem[HS93]{HunekeSharp}
Craig~L Huneke and Rodney~Y Sharp.
\newblock Bass numbers of local cohomology modules.
\newblock {\em Transactions of the American Mathematical Society},
  339(2):765--779, 1993.

\bibitem[Kun69]{Kunz}
Ernst Kunz.
\newblock Characterizations of regular local rings for characteristic {$p$}.
\newblock {\em Amer. J. Math.}, 91:772--784, 1969.

\bibitem[KZ14]{KatzmanZhang}
Mordechai Katzman and Wenliang Zhang.
\newblock Castelnuovo-{M}umford regularity and the discreteness of
  {$F$}-jumping coefficients in graded rings.
\newblock {\em Trans. Amer. Math. Soc.}, 366(7):3519--3533, 2014.

\bibitem[Lyu93]{LyuDMod}
Gennady Lyubeznik.
\newblock Finiteness properties of local cohomology modules (an application of
  d-modules to commutative algebra).
\newblock {\em Inventiones mathematicae}, 113(1):41--55, 1993.

\bibitem[Lyu97]{LyuFmod}
Gennady Lyubeznik.
\newblock F-modules: applications to local cohomology and d-modules in
  characteristic $p> 0$.
\newblock {\em Journal f{\"u}r die reine und angewandte Mathematik (Crelles
  Journal)}, 1997(491):65--130, 1997.

\bibitem[Lyu06]{LyuVan}
Gennady Lyubeznik.
\newblock On the vanishing of local cohomology in characteristic {$p>0$}.
\newblock {\em Compos. Math.}, 142(1):207--221, 2006.

\bibitem[Ma13]{MaRFmod}
Linquan Ma.
\newblock Finiteness properties of local cohomology for {$F$}-pure local rings.
\newblock {\em International Mathematics Research Notices}, 2013.

\bibitem[Ma15]{MaBuchsbaum}
Linquan Ma.
\newblock {$F$}-injectivity and {B}uchsbaum singularities.
\newblock {\em Math. Ann.}, 362(1-2):25--42, 2015.

\bibitem[Mat12]{Matsuda}
Kazunori Matsuda.
\newblock Weakly closed graphs and f-purity of binomial edge ideals.
\newblock {\em arXiv:1209.4300}, 2012.

\bibitem[NBWZ13]{SurveyLyuNum}
Luis N{\'u}{\~n}ez-Betancourt, Emily~E. Witt, and Wenliang Zhang.
\newblock A survey on the {L}yubeznik numbers.
\newblock {\em Preprint}, 2013.

\bibitem[Ogu73]{O}
Arthur Ogus.
\newblock Local cohomological dimension of algebraic varieties.
\newblock {\em Ann. of Math. (2)}, 98:327--365, 1973.

\bibitem[PS73]{P-S}
Christian Peskine and Lucien Szpiro.
\newblock Dimension projective finie et cohomologie locale.
\newblock {\em Publications Math{\'e}matiques de l'IH{\'E}S}, 42(1):47--119,
  1973.

\bibitem[Sch82]{Schenzel}
Peter Schenzel.
\newblock Applications of dualizing complexes to {B}uchsbaum rings.
\newblock {\em Adv. in Math.}, 44(1):61--77, 1982.

\bibitem[Sch10]{KarlCentersFpurity}
Karl Schwede.
\newblock Centers of {$F$}-purity.
\newblock {\em Math. Z.}, 265(3):687--714, 2010.

\bibitem[SW05]{AnuragUliSegre}
Anurag Singh and Uli Walther.
\newblock On the arithmetic rank of certain {S}egre products.
\newblock {\em Contemporary Mathematics}, page 147Ð155, 2005.

\bibitem[SW07]{AnuragUliPure}
Anurag~K. Singh and Uli Walther.
\newblock Local cohomology and pure morphisms.
\newblock {\em Illinois J. Math.}, 51(1):287--298 (electronic), 2007.

\bibitem[Wal99]{Walther}
Uli Walther.
\newblock Algorithmic computation of local cohomology modules and the local
  cohomological dimension of algebraic varieties.
\newblock {\em J. Pure Appl. Algebra}, 139(1-3):303--321, 1999.
\newblock Effective methods in algebraic geometry (Saint-Malo, 1998).

\bibitem[Wal01]{Walther2}
Uli Walther.
\newblock On the {L}yubeznik numbers of a local ring.
\newblock {\em Proc. Amer. Math. Soc.}, 129(6):1631--1634 (electronic), 2001.

\bibitem[Yan01]{YStr}
Kohji Yanagawa.
\newblock Bass numbers of local cohomology modules with supports in monomial
  ideals.
\newblock {\em Math. Proc. Cambridge Philos. Soc.}, 131(1):45--60, 2001.

\bibitem[Yas09]{Yasuda}
Takehiko Yasuda.
\newblock On monotonicity of {$F$}-blowup sequences.
\newblock {\em Illinois J. Math.}, 53(1):101--110, 2009.

\bibitem[Zha07]{W}
Wenliang Zhang.
\newblock On the highest {L}yubeznik number of a local ring.
\newblock {\em Compos. Math.}, 143(1):82--88, 2007.

\bibitem[Zha11]{WZ-projective}
Wenliang Zhang.
\newblock Lyubeznik numbers of projective schemes.
\newblock {\em Adv. Math.}, 228(1):575--616, 2011.

\end{thebibliography}
\end{document}